\documentclass[11pt,a4paper]{article}
\usepackage[T1]{fontenc}
\usepackage{lmodern,microtype}
\usepackage[margin=27mm]{geometry}
\usepackage{amsmath,amssymb,amsthm,mathtools}
\usepackage[numbers,sort&compress]{natbib}
\usepackage{xurl}
\usepackage[hidelinks]{hyperref}
\allowdisplaybreaks
\numberwithin{equation}{section}
\newtheorem{theorem}{Theorem}[section]
\newtheorem{proposition}[theorem]{Proposition}
\newtheorem{lemma}[theorem]{Lemma}
\newtheorem{corollary}[theorem]{Corollary}
\theoremstyle{definition}
\newtheorem{hypothesis}[theorem]{Hypothesis}
\theoremstyle{remark}

\newcommand{\dd}{\,\mathrm{d}}
\newcommand{\ii}{\mathrm{i}}
\newcommand{\one}{\mathbf{1}}
\newcommand{\R}{\mathbb{R}}

\newcommand{\Sg}{\mathfrak{S}}
\newcommand{\Rea}{\operatorname{Re}}
\newcommand{\Ima}{\operatorname{Im}}
\newcommand{\supp}{\operatorname{supp}}
\newcommand{\Mell}[1]{\widetilde{#1}}

\title{The Hardy--Littlewood asymptotic\\
for prime pairs of fixed even difference\\
from a uniform zeta-ratios conjecture}
\author{SeungJu Lee}
\date{\today}

\begin{document}
\maketitle

\begin{abstract}
We prove the Hardy--Littlewood asymptotic for prime pairs of every fixed
positive even difference \(h\) from a restricted uniform two-over-two
zeta-ratios formula. The formula may have error
\(O_\varepsilon(U^{\vartheta+\varepsilon})\) for any fixed
\(1/2\le\vartheta<(39-\sqrt{177})/32\). The real parts of all four
shifts lie in \([3/64,5/64]\); their imaginary parts need only have
absolute value at most \(U^\nu\), for a fixed \(0<\nu\le1/4\).
For a fixed real smooth weight supported in
\((1,1+1/(25h))\), we obtain an error
\(O_{h,W,\varepsilon}(x^{-\Delta(\vartheta)+\varepsilon})\)
in the correlation sum with weight \(W(n/x)^2/n\), where
\(\Delta(\vartheta)=7/16-\vartheta+\vartheta/(3-\vartheta)>0\).
Adding the Riemann hypothesis improves this error to
\(O_{h,W,\varepsilon}(x^{\vartheta-7/8+\varepsilon})\).
The proof retains the crossed zero residues in Mellin inversion and
estimates their mixed term with the logarithmic-derivative integral.
Two nonnegative time weights have the same diagonal term and opposite
coefficients for the difference \(h\). We deduce the logarithmic
asymptotic for the weighted prime-pair sum \(B_h\) and an asymptotic
for the second finite difference of integrals of squared increments of
the truncated Chebyshev function. The latter is connected to a
logarithmically weighted M\"obius approximant in the Nyman--Beurling
criterion by finite identities.
\end{abstract}

\section{Prime-pair asymptotics and the improvement under RH}
\label{sec:main}

Throughout, \(p\) denotes a prime, \(\Lambda\) is the von Mangoldt
function, \(\mu\) is the M\"obius function, and all logarithms are
natural. Fix a positive even integer \(h\), and put
\[
 \kappa_h=1+\frac1{25h},\qquad
 \vartheta_* = \frac{39-\sqrt{177}}{32},\qquad
 \Delta(\vartheta)=\frac7{16}-\vartheta+\frac{\vartheta}{3-\vartheta}.
\]
Throughout, \(1/2\le\vartheta<\vartheta_*\) and \(0<\nu\le1/4\)
are fixed. In this range,
\(0<\Delta(\vartheta)<7/8-\vartheta\). Define
\begin{equation}\label{eq:C2}
 C_2=\prod_{p>2}\left(1-\frac{1}{(p-1)^2}\right),
 \qquad \Sg(h)=2C_2\prod_{\substack{p\mid h\\p>2}}\frac{p-1}{p-2}.
\end{equation}
Every factor is positive and the sum of the deficits is finite, so
\(C_2>0\) and \(\Sg(h)>0\). For real \(x\ge2\), put
\begin{equation}\label{eq:counting-defs}
 A_h(x)=\sum_{\substack{p\le x\\p+h\ {\rm prime}}}
                   \log p\log(p+h),
 \qquad
 \pi_h(x)=\#\{p\le x:p+h\ {\rm prime}\}.
\end{equation}
The predicted asymptotic for \(\pi_h\) is the case \(k=h\) of
Hardy--Littlewood's Conjecture~B
\cite[\S5.2, (5.26) and Conjecture~B]{HardyLittlewood1923}.

\begin{theorem}\label{thm:main}
Assume Hypothesis~\ref{hyp:ratios}. Then, as \(x\to\infty\),
\begin{align}
 \sum_{n\le x}\Lambda(n)\Lambda(n+h)&=\Sg(h)x+o(x),
       \label{eq:main-Lambda}\\
 A_h(x)&=\Sg(h)x+o(x),
 \qquad
 \pi_h(x)\sim \Sg(h)\frac{x}{(\log x)^2}.
       \label{eq:main}
\end{align}
\end{theorem}

The quantitative input is the following estimate for fixed smooth
weights.

\begin{theorem}\label{thm:smoothed}
Assume Hypothesis~\ref{hyp:ratios}, and let
\(W\in C_c^\infty((1,\kappa_h))\) be real and fixed. For each fixed
\(0<\varepsilon<\Delta(\vartheta)\),
\begin{equation}\label{eq:smoothed-main}
 \sum_{n\ge1}\frac{\Lambda(n)\Lambda(n+h)}n W(n/x)^2
 =\Sg(h)\int_0^\infty W(u)^2\frac{\dd u}{u}
       +O_{h,W,\varepsilon}(x^{-\Delta(\vartheta)+\varepsilon}).
\end{equation}
If RH is also assumed, then for each fixed
\(0<\varepsilon<7/8-\vartheta\) the error can be replaced by
\begin{equation}\label{eq:smoothed-RH-error}
 O_{h,W,\varepsilon}(x^{\vartheta-7/8+\varepsilon}).
\end{equation}
Both statements remain valid when the sum on the left is replaced by
\begin{equation}\label{eq:smoothed-prime-sum}
 \sum_{\substack{p\\p+h\ {\rm prime}}}
       \frac{\log p\log(p+h)}p W(p/x)^2.
\end{equation}
\end{theorem}

For example, \(\vartheta=3/4\) gives the errors
\(O_{h,W,\varepsilon}(x^{-1/48+\varepsilon})\) and, under RH,
\(O_{h,W,\varepsilon}(x^{-1/8+\varepsilon})\).
The choice \(\vartheta=4/5\) gives
\(O_{h,W,\varepsilon}(x^{-1/880+\varepsilon})\) and
\(O_{h,W,\varepsilon}(x^{-3/40+\varepsilon})\), respectively.
For the square-root ratios error \(\vartheta=1/2\), they become
\(O_{h,W,\varepsilon}(x^{-11/80+\varepsilon})\) and
\(O_{h,W,\varepsilon}(x^{-3/8+\varepsilon})\).

Section~\ref{sec:unsmoothing} removes smoothing from
\eqref{eq:smoothed-main} to obtain the \(o(x)\) remainders in
Theorem~\ref{thm:main}.

For real \(x\ge2\), define
\begin{equation}\label{eq:Bh}
 B_h(x)=\sum_{\substack{p+h\le x\\p,\ p+h\ {\rm prime}}}
       \log p\log(p+h)\left(\frac1{p+h}-\frac1x\right).
\end{equation}

\begin{corollary}\label{cor:series}
Under Hypothesis~\ref{hyp:ratios},
\begin{equation}\label{eq:series-main}
 B_h(x)=\Sg(h)\log x+o(\log x).
\end{equation}
For every \(0<c<\Sg(h)\) there is \(x_c\) such that
\(B_h(x)\ge c\log x\) for every real \(x\ge x_c\).
\end{corollary}

Section~\ref{sec:NB} derives \eqref{eq:energy-asymptotic} from
\eqref{eq:main-Lambda} through an exact second-difference identity for
integrals of squared increments of the truncated Chebyshev function.
That function is expressed there through a finite Nyman--Beurling
approximant.

Hypothesis~\ref{hyp:ratios} uses the two-over-two formula of Conrey
and Snaith \cite[Conjecture~2.1]{ConreySnaith2007}, restricted to an
explicit set of shifts, with a relaxed error bound, and treated here
as a standalone hypothesis about finite integrals. Conrey and Snaith assume RH throughout their
paper \cite[\S1, final paragraph]{ConreySnaith2007}.
We derive the zero-free strip~\eqref{eq:zero-free-band} and the
zero-count bound~\eqref{eq:ratios-density} from the displayed ratios
hypothesis and combine them with an unconditional estimate of Ingham.

There are precedents for relating detailed zeta statistics to
prime-pair correlations. Conrey and Keating
\cite[\S\S2--4]{ConreyKeating2016} study the relation between a
uniform ratios conjecture and correlations of Dirichlet coefficients
through means of truncated ratios. Their arithmetic calculation uses
the generalized Riemann hypothesis to identify the poles of twisted
Dirichlet series. Keating and Smith
\cite[Introduction and \S4.2]{KeatingSmith2019} formally recover the
Hardy--Littlewood prediction from a two-point correlation formula that
retains its arithmetic terms. Here we use a quantitative differentiated
ratios formula and Fourier support that selects the difference \(h\).

Mellin inversion gives \(P_x=Q_x+R_x-Z_x\), where \(R_x\) is the
contribution of the pole of \(\zeta\) and \(Z_x\) sums the crossed
zeros. We estimate the mixed term involving \(Q_xZ_x\). For each
of two nonnegative weights, the term \(|R_x-Z_x|^2\) has a favorable
sign, and the resulting inequalities give opposite bounds for the
same prime-pair sum. Under RH, \(Z_x=0\), giving the improved error
in \eqref{eq:smoothed-RH-error}.

Implicit constants may depend on the fixed parameters \(h,\vartheta,\nu\),
on fixed test functions and on parameters displayed as subscripts.
Every limit keeps these parameters and its test functions fixed.

\section{The ratios formula and a bound for zeros to the right of \texorpdfstring{\(9/16\)}{9/16}}
\label{sec:ratios}
For four complex shifts, put
\begin{align}
 \mathcal A(\alpha,\beta;\gamma,\delta)
 &=\prod_p
 \frac{(1-p^{-1-\gamma-\delta})
       (1-p^{-1-\beta-\gamma}-p^{-1-\alpha-\delta}
                         +p^{-1-\gamma-\delta})}
      {(1-p^{-1-\beta-\gamma})(1-p^{-1-\alpha-\delta})},
 \label{eq:ratio-Euler}\\
 \mathcal Y(\alpha,\beta;\gamma,\delta)
 &=\frac{\zeta(1+\alpha+\beta)\zeta(1+\gamma+\delta)}
         {\zeta(1+\alpha+\delta)\zeta(1+\beta+\gamma)}
   \mathcal A(\alpha,\beta;\gamma,\delta).
 \label{eq:ratio-Y}
\end{align}
The Euler product converges absolutely in the ranges used below.
Expressions at removable singularities are interpreted by continuation.

\begin{hypothesis}[Uniform two-over-two ratios formula]\label{hyp:ratios}
For every \(\varepsilon>0\), as \(U\to\infty\),
\begin{align}
 &\int_0^U
 \frac{\zeta(1/2+\alpha+\ii t)\zeta(1/2+\beta-\ii t)}
      {\zeta(1/2+\gamma+\ii t)\zeta(1/2+\delta-\ii t)}
 \dd t \notag\\
 &\quad =
 U\mathcal Y(\alpha,\beta;\gamma,\delta)
 +\int_0^U\left(\frac{t}{2\pi}\right)^{-\alpha-\beta}
       \mathcal Y(-\beta,-\alpha;\gamma,\delta)\dd t
 +O_\varepsilon(U^{\vartheta+\varepsilon}),
 \label{eq:ratios-hyp}
\end{align}
uniformly for
\begin{equation}\label{eq:ratio-shifts}
 \frac{3}{64}\le\Rea\alpha,\Rea\beta,\Rea\gamma,\Rea\delta
       \le\frac{5}{64},\qquad
 |\Ima\alpha|,|\Ima\beta|,|\Ima\gamma|,|\Ima\delta|\le U^\nu.
\end{equation}
The integral is an ordinary finite integral; only removable
singularities in the explicit expressions are filled by continuation.
\end{hypothesis}

The main terms are those in \cite[(2.8), (2.10)]{ConreySnaith2007}, restricted
to \eqref{eq:ratio-shifts}. The source gives its shift conditions in
(2.11a--c), states the uniformity of the error immediately after them,
and discusses differentiation in Remark~2.4. Here we allow the
error exponent \(\vartheta\) and require uniformity only in the
smaller imaginary-shift range \(U^\nu\).

Set
\begin{equation}\label{eq:eta}
 \eta=\frac1{16},\qquad \sigma_0=\frac12+\eta=\frac9{16}.
\end{equation}
For \(T\ge2\), let \(N_{>}(\sigma,T)\) count nontrivial zeros
\(\rho=\beta+\ii\gamma\) with \(\beta>\sigma\) and
\(0<\gamma\le T\), including multiplicities.

\begin{lemma}\label{lem:zeros}
Hypothesis~\ref{hyp:ratios} implies that \(\zeta\) has no zeros in
\begin{equation}\label{eq:zero-free-band}
 \frac{35}{64}\le\Rea s\le\frac{37}{64}.
\end{equation}
It also implies
\begin{align}
 \frac{\zeta'}{\zeta}(\sigma_0+\ii v)
   &\ll \log^2(2+|v|),\qquad v\in\R,
     \label{eq:logderivative-bound}\\
 N_{>}(\sigma_0,T)&\ll_\varepsilon T^{\vartheta+\varepsilon}
       \qquad(\varepsilon>0).
     \label{eq:ratios-density}
\end{align}
\end{lemma}

\begin{proof}
Suppose \(\rho=1/2+d+\ii\gamma_0\) is a zero, where
\(3/64\le d\le5/64\). Zeros in this strip have nonzero imaginary
part, and conjugation permits us to take \(\gamma_0>0\). Choose a
real \(a\in[3/64,5/64]\) such that
\(\zeta(1/2+a+\ii\gamma_0)\ne0\); isolated zeros ensure such a
choice. Set \(\alpha=\beta=a\) and \(\gamma=\delta=d\) in
\eqref{eq:ratios-hyp}. Its integrand is
\[
 \left|\frac{\zeta(1/2+a+\ii t)}
                  {\zeta(1/2+d+\ii t)}\right|^2.
\]
If the zero has multiplicity \(m\), this is a positive constant
times \(|t-\gamma_0|^{-2m}\), to leading order. The integral is
infinite for every \(U>\gamma_0\). The explicit main terms are
finite for these real shifts: the zeta arguments in the reciprocal
factors are positive real numbers, and the pole at \(1\) gives a
zero of the reciprocal. This contradicts the hypothesis and proves
\eqref{eq:zero-free-band}.

The distance in real part from \(\sigma_0\) to any zero is at
least \(1/64\). The partial-fraction estimate for \(\zeta'/\zeta\)
and the zero count in intervals of length one
\cite[Theorems~9.2 and~9.6(A), pp.~211 and~217]{Titchmarsh1986}
therefore give \eqref{eq:logderivative-bound}. The compact range of
heights is covered by continuity.

Next set \(\beta=\gamma=\delta=\eta\) in
\eqref{eq:ratios-hyp}, leaving \(\alpha\) free near \(\eta\).
The second ratio cancels. Direct substitution in
\eqref{eq:ratio-Euler}--\eqref{eq:ratio-Y} gives
\[
 \mathcal Y(\alpha,\eta;\eta,\eta)=1,
 \qquad
 \mathcal Y(-\eta,-\alpha;\eta,\eta)=0.
\]
The second identity uses the factor \(1/\zeta(1)\), with its value
given by analytic continuation. All other factors are regular near
\(\alpha=\eta\). Consequently
\[
 \int_0^U\frac{\zeta(1/2+\alpha+\ii t)}
                    {\zeta(\sigma_0+\ii t)}\dd t
 =U+O_\varepsilon(U^{\vartheta+\varepsilon}).
\]
The denominator is nonzero by \eqref{eq:zero-free-band}. Cauchy's
formula on the circle \(|\alpha-\eta|=1/128\), which lies in the
allowed shift range for large \(U\), gives
\begin{equation}\label{eq:first-derivative}
 \int_0^U\frac{\zeta'}{\zeta}(\sigma_0+\ii t)\dd t
 =O_\varepsilon(U^{\vartheta+\varepsilon}).
\end{equation}

Apply the argument principle to a rectangle with
vertical sides \(\sigma_0\) and \(2\), a fixed lower height
\(t_0>0\) containing no zero on its horizontal side, and an upper
height \(T'\in[T,T+1]\). The latter can be chosen at distance
\(\gg1/\log(2+T)\) from every zero ordinate. Indeed, only
\(O(\log(2+T))\) zeros, counted with multiplicity, have ordinates
in \([T-1,T+2]\); exclude sufficiently short intervals around them.
The cited partial-fraction estimate then bounds the integral on the
upper side by \(O(\log^2(2+T))\). On the right side the absolutely
convergent series gives
\[
 \int_{t_0}^{T'}\frac{\zeta'}{\zeta}(2+\ii t)\dd t=O(1),
\]
since \(\sum_{n\ge2}\Lambda(n)/(n^2\log n)<\infty\).
The lower side contributes a fixed constant, and the left side is
bounded by \eqref{eq:first-derivative}. The zeta-function has no poles
inside this rectangle. The argument principle thus bounds its zero count
by \(O_\varepsilon(T^{\vartheta+\varepsilon})\). Monotonicity and the
finitely many lower zeros prove \eqref{eq:ratios-density}.
\end{proof}

We also use the unconditional Ingham estimate
\cite[Theorem~9.19(B), p.~236]{Titchmarsh1986}:
\begin{equation}\label{eq:Ingham}
 N_{>}(\sigma,T)\ll T^{3(1-\sigma)/(2-\sigma)}\log^5T,
 \qquad \frac12<\sigma<1.
\end{equation}
\subsection{The differentiated two-ratio mean}
For \(\Rea z>0\), define
\begin{align}
 D(z)&=\sum_{n\ge1}\frac{\Lambda(n)^2}{n^{1+z}},
 \label{eq:D-def}\\
 A(z)&=\prod_p\left(1-\frac{(p^{-z}-1)^2}{(p-1)^2}\right),
 \qquad
 C(z)=\zeta(1+z)\zeta(1-z)A(z).
 \label{eq:C-def}
\end{align}
The product \(A(z)\) converges absolutely for \(\Rea z>0\) and is
bounded on each closed vertical strip in that half-plane.

\begin{lemma}\label{lem:logderivative}
Assume Hypothesis~\ref{hyp:ratios}. If
\(\Rea a=\Rea b=\eta\), \(z=a+b\), and
\(|\Ima a|,|\Ima b|\le U^{\nu/2}\), then
\begin{align}
 \mathcal F_{a,b}(U)
 &:=\int_0^U\frac{\zeta'}{\zeta}(1/2+a+\ii t)
                 \frac{\zeta'}{\zeta}(1/2+b-\ii t)\dd t
 \notag\\
 &=UD(z)+(2\pi)^zC(z)\frac{U^{1-z}}{1-z}
       +O_\varepsilon(U^{\vartheta+\varepsilon}),
 \label{eq:logderivative}
\end{align}
with an error uniform in \(a,b\).
\end{lemma}

\begin{proof}
Differentiate \eqref{eq:ratios-hyp} in \(\alpha\) and \(\beta\), keeping
\(\gamma=a,\delta=b\) fixed, and then set \(\alpha=a,\beta=b\).
The differentiation can be performed on circles of radius \(1/128\).
Their real parts lie in \([3/64,5/64]\), and their imaginary parts
have absolute value at most \(U^{\nu/2}+1/128<U^\nu\) for sufficiently
large \(U\). Lemma~\ref{lem:zeros} makes the denominator zeta functions nonzero.
The reciprocal zeta factors in the swapped term are analytic near
their arguments \(1\). The denominator zeta factors in the
unswapped term have arguments with real part greater than \(1\).
Cauchy's formula therefore preserves the stated error up to a fixed
constant.

The differentiated main terms are those in
\cite[Theorem~2.5, (2.12)]{ConreySnaith2007}. Their term independent of
\(t\) is
\[
 \left(\frac{\zeta'}{\zeta}\right)'(1+z)
 -\sum_p\left(\frac{\log p}{p^{1+z}-1}\right)^2=D(z).
\]
For the coefficient of \(p^{-k(1+z)}(\log p)^2\), the two contributions
are \(k\) and \(k-1\), with the second equal to zero when \(k=1\).
The remaining term is \((t/2\pi)^{-z}C(z)\), since
\[
 \frac{(1-p^{-1-z})(1-2/p+p^{-1-z})}{(1-1/p)^2}
 =1-\frac{(p^{-z}-1)^2}{(p-1)^2}.
\]
Integrating \(t^{-z}\), with \(\Rea z=1/8<1\), proves
\eqref{eq:logderivative}.
\end{proof}

\section{Two nonnegative weights with opposite coefficients at difference \texorpdfstring{\(h\)}{h}}
\label{sec:Fourier}

Use the Fourier transform convention
\begin{equation}\label{eq:Fourier-convention}
 \widehat f(v)=\int_\R f(t)e^{-2\pi\ii vt}\dd t.
\end{equation}
Choose a real even \(b\in C_c^\infty(\R)\) that is positive for
\(|v|<1/40\) and zero for \(|v|\ge1/40\). Let \(k\) be its
inverse Fourier transform, and set
\begin{equation}\label{eq:weight-construction}
 \ell(t)=\sin(\pi t/5)k(t),\quad
 H=\widehat{\ell^2},\quad
 g_\pm(t)=\ell(t)^2(1\pm\cos(2\pi ht)).
\end{equation}
Here \(\ell\) is real, odd and Schwartz. Both \(g_+\) and \(g_-\)
are real, even, nonnegative Schwartz functions, and
\begin{equation}\label{eq:g-zero}
 g_\pm(0)=g_\pm'(0)=0,\qquad
 g_\pm(t)=O(t^2),\quad g_\pm'(t)=O(t)\quad(t\to0).
\end{equation}
Writing \(K=b*b\), elementary Fourier identities give
\begin{align}
 H(v)&=\frac12K(v)-\frac14K(v-1/5)-\frac14K(v+1/5),
       \label{eq:H}\\
 \phi_\pm(v):=\widehat g_\pm(v)
     &=H(v)\pm\frac12H(v-h)\pm\frac12H(v+h).
       \label{eq:phi-pm}
\end{align}
The function \(K\) is positive on \((-1/20,1/20)\) and supported
in \([-1/20,1/20]\). Hence
\begin{equation}\label{eq:H-support}
 \supp H\subset[-1/4,1/4],\quad
 H(v)>0\quad(|v|<1/20),\quad H(0)>0.
\end{equation}
In particular, both Fourier transforms have the same value \(H(0)\)
at zero and support in
\[
 [-1/4,1/4]\cup[h-1/4,h+1/4]\cup[-h-1/4,-h+1/4].
\]
On the two side intervals they satisfy \(\phi_-=-\phi_+\). Define
\begin{equation}\label{eq:a-weight}
 a(u)=\frac12H(h/u-h),\qquad u>0.
\end{equation}
This is smooth on \((0,\infty)\) and strictly positive on
\((1,\kappa_h)\), since there
\(-h/(25h+1)<h/u-h<0\) and \(h/(25h+1)<1/25<1/20\).

Fix a real \(W\in C_c^\infty((1,\kappa_h))\), and set
\begin{align}
 T&=2\pi x,&
 P_x(t)&=\sum_{n\ge1}\frac{\Lambda(n)}{n^{1/2+\ii t}}W(n/x),
 \label{eq:polynomial}\\
 I_\pm(x)&=\int_0^\infty g_\pm(t/T)|P_x(t)|^2\dd t,&
 D_x&=\frac{T}{2}H(0)\sum_n\frac{\Lambda(n)^2}{n}W(n/x)^2.
 \label{eq:I-pm}
\end{align}

\begin{lemma}\label{lem:separate-gap}
For all sufficiently large \(x\),
\begin{equation}\label{eq:moment-arithmetic}
 I_\pm(x)=D_x\pm T J_W(x),
\end{equation}
where
\begin{equation}\label{eq:J}
 J_W(x)=\frac12\sum_n
 \frac{\Lambda(n)\Lambda(n+h)}{\sqrt{n(n+h)}}
 W(n/x)W((n+h)/x)
 H\bigl(x\log(1+h/n)-h\bigr).
\end{equation}
\end{lemma}

\begin{proof}
The function \(|P_x(t)|^2\) is even. Expanding the finite sum and
using \eqref{eq:Fourier-convention} gives, for either sign,
\[
 I_\pm(x)=D_x+T\sum_{j\ge1}\sum_n
 \frac{\Lambda(n)\Lambda(n+j)}{\sqrt{n(n+j)}}
 W(n/x)W((n+j)/x)\phi_\pm\bigl(x\log(1+j/n)\bigr).
\]
When both \(W\) factors are nonzero, \(x<n<\kappa_h x\).
For large \(x\) the argument is greater than \(1/4\) for every
\(j\ge1\). For \(j\le h-1\), it is less than \(h-1\).
For \(j\ge h+1\), it is at least
\[
 x\log\left(1+\frac{h+1}{\kappa_h x}\right)>h+\frac14
\]
for sufficiently large \(x\), since
\((h+1)/\kappa_h-h=24h/(25h+1)>1/4\).
For \(j=h\), it lies in \((h-1/4,h+1/4)\), since
\(h-h/\kappa_h<1/25\). Thus only \(j=h\) survives, where
\(\phi_\pm(v)=\pm H(v-h)/2\). This proves the identity.
\end{proof}

\section{Mellin inversion with the crossed zero residues}
\label{sec:mean-square}

Write
\[
 \Mell W(s)=\int_0^\infty W(u)u^{s-1}\dd u
\]
and let \(\int_{(c)}\) denote upward integration on \(\Rea s=c\).
The transform \(\Mell W\) is entire and rapidly decreasing on
every fixed closed vertical strip. Define
\begin{align}
 Q_x(t)&=-\frac1{2\pi\ii}\int_{(\eta)}
       \frac{\zeta'}{\zeta}(1/2+\ii t+s)
       x^s\Mell W(s)\dd s,\label{eq:Q}\\
 R_x(t)&=x^{1/2-\ii t}\Mell W(1/2-\ii t),\label{eq:R}\\
 Z_x(t)&=\sum_{\substack{\rho=\beta+\ii\gamma\\\beta>\sigma_0}}
       x^{\rho-1/2-\ii t}\Mell W(\rho-1/2-\ii t).
       \label{eq:Z}
\end{align}
The sum in \eqref{eq:Z} includes all positive and negative ordinates
and counts multiplicities. It converges absolutely for each fixed
\(x,t\), by the rapid decrease of \(\Mell W\) and the standard
zero count.

\begin{lemma}\label{lem:contour}
Under Hypothesis~\ref{hyp:ratios},
\begin{equation}\label{eq:shifted-Mellin}
 P_x(t)=Q_x(t)+R_x(t)-Z_x(t).
\end{equation}
For every fixed \(M>0\),
\begin{align}
 |Q_x(t)|&\ll_W x^\eta\log^2(2+|t|),
       \label{eq:Q-bound}\\
 |R_x(t)|&\ll_{W,M}x^{1/2}(1+|t|)^{-M},
       \label{eq:R-bound}\\
 |Z_x(t)|&\ll_{W,M}
  \sum_{\beta>\sigma_0}x^{\beta-1/2}(1+|t-\gamma|)^{-M}.
       \label{eq:Z-bound}
\end{align}
\end{lemma}

\begin{proof}
Start with Mellin inversion on \(\Rea s=c>1/2\), where the
Dirichlet series for \(-\zeta'/\zeta(1/2+\ii t+s)\) converges
absolutely, and move to \(\Rea s=\eta\). The latter line is free
of poles by Lemma~\ref{lem:zeros}. The pole of \(\zeta\) contributes
\(R_x\). Each crossed zero contributes its negative multiplicity
times the corresponding summand in \eqref{eq:Z}. There are no
crossed trivial zeros. Choose the horizontal sides at heights away
from the zero ordinates, as in Lemma~\ref{lem:zeros}; the
partial-fraction estimate and rapid Mellin decay make those sides
tend to zero. This proves \eqref{eq:shifted-Mellin}.
Estimate \eqref{eq:Q-bound} follows from
\eqref{eq:logderivative-bound} and rapid Mellin decay; the other two
estimates follow directly from their definitions, uniformly over
\(\sigma_0<\beta<1\).
\end{proof}

\subsection{The ratios evaluation of the mean of \texorpdfstring{\(Q_x\)}{Qx}}

In this section and the next, \(g\) denotes either \(g_+\) or \(g_-\),
and \(\phi=\widehat g\). Put
\begin{equation}\label{eq:Gmellin}
 G(s)=\int_0^\infty g(v)v^{s-1}\dd v,\qquad \Rea s>-2.
\end{equation}
By \eqref{eq:g-zero}, this is analytic in the indicated half-plane
and rapidly decreasing on each closed vertical strip in it.

\begin{proposition}\label{prop:moment-formula}
Under Hypothesis~\ref{hyp:ratios}, for each fixed
\(0<\varepsilon<7/8-\vartheta\),
\begin{align}
 \int_0^\infty g(t/T)|Q_x(t)|^2\dd t
 &=\frac{T}{2}\phi(0)\sum_n\frac{\Lambda(n)^2}{n}W(n/x)^2
       \notag\\
 &\quad+\frac{T}{2\pi\ii}\int_{(2\eta)}
       \Mell{W^2}(z)C(z)G(1-z)\dd z
       +O_{W,g,\varepsilon}(x^{\vartheta+1/8+\varepsilon}).
       \label{eq:moment-formula}
\end{align}
\end{proposition}
\subsection{Integration of the uniform ratios error}
\label{subsec:ratios-error}

Set
\begin{equation}\label{eq:truncations}
 K=x^{\nu/4},\qquad U_0=x^{1/2}.
\end{equation}
Truncate the two Mellin lines obtained from \(|Q_x(t)|^2\) at
imaginary height \(K\). The estimate \eqref{eq:logderivative-bound}
and rapid decrease of \(\Mell W\) show that, for every fixed \(M>0\),
the resulting change in the time integral is
\begin{equation}\label{eq:Mellin-tail}
 O_{W,g,M}\bigl(Tx^{2\eta}K^{-M}\log^4(2+T+K)\bigr).
\end{equation}
The exponent \(M\) can be chosen so that this is \(O(x^{-1})\).
Conjugating the second Mellin integral and reversing its parameter
expresses the remaining moment using shifts \(a,b\) with
\(\Rea a=\Rea b=\eta\) and
\(|\Ima a|,|\Ima b|\le K\).

Let
\[
 E_{a,b}(U)=\mathcal F_{a,b}(U)
       -UD(a+b)
       -(2\pi)^{a+b}C(a+b)\frac{U^{1-a-b}}{1-a-b}.
\]
Integration by parts expresses the error in the smoothed time
integral as
\begin{equation}\label{eq:Stieltjes-error}
 -\frac{1}{T}\int_0^\infty g'(U/T)E_{a,b}(U)\dd U.
\end{equation}
The boundary terms vanish. At zero, both explicit powers of \(U\)
tend to zero; at infinity, the rapid decrease of \(g\) applies.

For \(U\ge U_0\), the truncation satisfies \(K\le U^{\nu/2}\), so
Lemma~\ref{lem:logderivative} gives
\[
 \frac{1}{T}\int_{U_0}^\infty
       |g'(U/T)E_{a,b}(U)|\dd U
 \ll_{g,\varepsilon}T^{\vartheta+\varepsilon}
       \int_0^\infty |g'(v)|v^{\vartheta+\varepsilon}\dd v
 \ll_{g,\varepsilon}T^{\vartheta+\varepsilon}.
\]
This bound is uniform on the truncated Mellin lines.

For \(0<U<U_0\), we use \eqref{eq:logderivative-bound} directly.
The function \(D(a+b)\) is bounded on these lines, while \(C(a+b)\)
has at most polynomial growth in \(|\Ima a|+|\Ima b|\).
The latter follows from \eqref{eq:C-def}, the boundedness of \(A(z)\)
on \(\Rea z=1/8\), and the polynomial growth of \(\zeta\) on fixed
vertical strips
\cite[\S5.1, p.~95, discussion following (5.1.1)]{Titchmarsh1986}.
After integration against the absolute values of the two Mellin
transforms, the part of \eqref{eq:Stieltjes-error} with \(U<U_0\)
is therefore at most
\begin{align}
 &O_{W,g}\left(
 x^{2\eta}T^{-2}
 \left\{U_0^3\log^4(2+U_0+K)+U_0^{3-2\eta}\right\}\right)
 \notag\\
 &\hspace{35mm}=O_{W,g}(x^{-3/8}\log^4x).
 \label{eq:small-time}
\end{align}
Here \(g'(U/T)=O(U/T)\) and the polynomial factors in the shifts
are integrable against the rapidly decreasing Mellin transforms.
The factor \(x^{2\eta}\) in \eqref{eq:small-time} and in the large-time
error comes from \(x^{a+b}\). Thus the total error from the ratios
formula is
\begin{equation}\label{eq:ratios-error-final}
 O_{W,g,\varepsilon}(x^{2\eta}T^{\vartheta+\varepsilon})
 =O_{W,g,\varepsilon}(x^{\vartheta+1/8+\varepsilon}).
\end{equation}

\subsection{The two main terms after Mellin inversion}
\label{subsec:main-terms}

The integrations in \(a,b\) can now be extended back to the full
lines. The error in doing so is bounded as in \eqref{eq:Mellin-tail},
because \(D\), \(C\) and \(G(1-a-b)\) have at most polynomial
growth there. The first main term is
\[
 TG(1)\frac{1}{(2\pi\ii)^2}
 \int_{(\eta)}\int_{(\eta)}
 x^{a+b}\Mell W(a)\Mell W(b)D(a+b)\dd a\dd b.
\]
Expanding the absolutely convergent series for \(D\) and applying
Mellin inversion twice gives
\[
 TG(1)\sum_n\frac{\Lambda(n)^2}{n}W(n/x)^2
 =\frac{T}{2}\phi(0)\sum_n\frac{\Lambda(n)^2}{n}W(n/x)^2.
\]
The equality \(G(1)=\phi(0)/2\) uses evenness of \(g\).

The time integral of the other main term is
\(T(2\pi/T)^{a+b}C(a+b)G(1-a-b)\).
The Mellin convolution identity
\[
 \frac{1}{2\pi\ii}\int_{(\eta)}
      \Mell W(a)\Mell W(z-a)\dd a=\Mell{W^2}(z),
 \qquad \Rea z=2\eta,
\]
and \(T=2\pi x\) give its contribution as
\[
 \frac{T}{2\pi\ii}\int_{(2\eta)}
      \Mell{W^2}(z)C(z)G(1-z)\dd z.
\]
All these integrals are absolutely convergent. Combining the two
main terms with \eqref{eq:Mellin-tail},
\eqref{eq:small-time} and \eqref{eq:ratios-error-final} proves
Proposition~\ref{prop:moment-formula}.

\section{The singular-series Dirichlet series and its Fourier inversion}
\label{sec:singular}
For positive integers \(k\), define
\begin{equation}\label{eq:singular-series}
 \Sg(k)=
 \begin{cases}
 2C_2\displaystyle\prod_{\substack{p\mid k\\p>2}}\dfrac{p-1}{p-2},
       &2\mid k,\\[5pt]
 0,&2\nmid k.
 \end{cases}
\end{equation}

\begin{lemma}\label{lem:singular-Dirichlet}
For \(\Rea z>1\), the absolutely convergent series
\(D_{\Sg}(z)=\sum_{k\ge1}\Sg(k)k^{-z}\) satisfies
\begin{equation}\label{eq:DS}
 D_{\Sg}(z)=2^{1-z}C_2\zeta(z)
       \prod_{p>2}\left(1+\frac{p^{-z}}{p-2}\right)
       =\zeta(z)\zeta(1+z)A(z).
\end{equation}
It extends meromorphically to \(\Rea z>0\), with only a simple pole
at \(z=1\) in that half-plane, of residue \(1\).
\end{lemma}

\begin{proof}
Writing \(k=2m\), the product over odd prime divisors of \(m\) is
multiplicative. Its prime factor at an odd \(p\) is
\[
 1+\frac{p-1}{p-2}\frac{p^{-z}}{1-p^{-z}}
 =\frac{1+p^{-z}/(p-2)}{1-p^{-z}},
\]
and its factor at \(2\) is \((1-2^{-z})^{-1}\).
This proves the first equality in \eqref{eq:DS}.
For an odd prime, the factor of \(\zeta(1+z)A(z)\) is
\[
 \frac{p(p-2+p^{-z})}{(p-1)^2}
 =\left(1-\frac{1}{(p-1)^2}\right)
      \left(1+\frac{p^{-z}}{p-2}\right);
\]
the factor at \(2\) is \(2^{1-z}\). This proves the second equality.
The continuation and the possible pole follow from the second
expression in \eqref{eq:DS}. At \(z=1\), each prime factor of
\(\zeta(2)A(1)\) equals \(1\), so the residue is \(1\).
\end{proof}

The functional equation \(\zeta(s)=\chi(s)\zeta(1-s)\), with
\[
 \chi(1-z)=2(2\pi)^{-z}\Gamma(z)\cos(\pi z/2),
\]
gives
\begin{equation}\label{eq:C-functional}
 C(z)=\chi(1-z)D_{\Sg}(z).
\end{equation}
Define also
\begin{equation}\label{eq:Phi}
 \Phi(z)=\int_0^\infty\phi(v)v^{z-1}\dd v,\qquad \Rea z>0.
\end{equation}
It is analytic and rapidly decreasing on closed vertical strips in
this half-plane.

\begin{lemma}\label{lem:Fourier-Mellin}
For \(0<\Rea z<3\),
\begin{equation}\label{eq:Fourier-Mellin}
 \Phi(z)=\chi(1-z)G(1-z).
\end{equation}
The product \(D_{\Sg}(z)\Phi(z)\) is analytic at \(z=1\).
\end{lemma}

\begin{proof}
Since \(g\) is real and even,
\(\phi(v)=2\int_0^\infty g(t)\cos(2\pi vt)\dd t\).
For \(0<\Rea z<1\), the Mellin transform of the cosine gives
\[
 \Phi(z)=2(2\pi)^{-z}\Gamma(z)\cos(\pi z/2)
             \int_0^\infty g(t)t^{-z}\dd t.
\]
To justify the interchange, insert \(e^{-av}\) with \(a>0\) in
the \(v\)-integral. The resulting cosine integral is
\[
 \frac{\Gamma(z)}{2}
 \bigl((a-2\pi\ii t)^{-z}+(a+2\pi\ii t)^{-z}\bigr),
\]
whose absolute value is bounded by a constant depending on \(z\)
times \(t^{-\Rea z}\), uniformly in \(a>0\).
Dominated convergence then applies, since \(g(t)=O(t^2)\) at zero
and decreases rapidly at infinity.
Both sides of \eqref{eq:Fourier-Mellin} are analytic in
\(0<\Rea z<3\), so the identity continues to that strip.
Finally,
\[
 \Phi(1)=\int_0^\infty\phi(v)\dd v=0
\]
because \(g(0)=0\) and Fourier inversion gives
\(\int_\R\phi(v)\dd v=g(0)\). This cancels the simple pole of
\(D_{\Sg}\) at \(1\).
\end{proof}

\begin{proposition}\label{prop:Q-evaluation}
Under Hypothesis~\ref{hyp:ratios}, put
\begin{equation}\label{eq:KWa}
 K_{W,a}=\int_0^\infty W(u)^2a(u)\frac{\dd u}{u}.
\end{equation}
Then, for every fixed \(0<\varepsilon<7/8-\vartheta\),
\begin{equation}\label{eq:Q-evaluated}
 \int_0^\infty g_\pm(t/T)|Q_x(t)|^2\dd t
 =D_x\pm T\Sg(h)K_{W,a}+O_{h,W,\varepsilon}(x^{\vartheta+1/8+\varepsilon}).
\end{equation}
\end{proposition}

\begin{proof}
By \eqref{eq:C-functional} and \eqref{eq:Fourier-Mellin}, the
integral term in \eqref{eq:moment-formula} is
\[
 \frac{T}{2\pi\ii}\int_{(2\eta)}
       \Mell{W^2}(z)D_{\Sg}(z)\Phi(z)\dd z.
\]
Move the line to \(\Rea z=3/2\). The zero \(\Phi(1)=0\) cancels
the only pole crossed. Polynomial growth of the zeta factors and
rapid decrease of the Mellin transforms justify the shift. On the
new line, insert the absolutely convergent Dirichlet series and
apply Mellin inversion. The result is
\begin{equation}\label{eq:singular-integral}
 T\sum_{j\ge1}\Sg(j)
       \int_0^\infty W(u)^2\phi_\pm(j/u)\frac{\dd u}{u}.
\end{equation}
For \(1<u<\kappa_h\), each \(j/u\) with \(j\ge1\) lies outside
the central support, and a side interval contains \(j/u\) only
when \(j=h\). There
\(\phi_\pm(h/u)=\pm a(u)\), so the expression is
\(\pm T\Sg(h)K_{W,a}\). The diagonal in
\eqref{eq:moment-formula} is \(D_x\), since
\(\phi_\pm(0)=H(0)\). This proves the claim.
\end{proof}

\section{Error bounds for smoothed prime-pair correlations}
\label{sec:mixed}

\subsection{The mixed zero and pole terms}

\begin{lemma}\label{lem:mixed}
Under Hypothesis~\ref{hyp:ratios}, for every fixed
\(\varepsilon>0\) and for \(g=g_+\) or \(g=g_-\),
\begin{align}
 \int_0^\infty |g(t/T)Q_x(t)Z_x(t)|\dd t
     &\ll_{W,g,\varepsilon}x^{1-\Delta(\vartheta)+\varepsilon},
       \label{eq:mixed-zero}\\
 \int_0^\infty |g(t/T)Q_x(t)R_x(t)|\dd t
     &\ll_{W,g}x^{-23/16}.
       \label{eq:mixed-pole}
\end{align}
\end{lemma}

\begin{proof}
Fix a large \(A>0\) and take \(M>A+3\). Schwartz decay gives,
uniformly in \(\gamma\in\R\) and \(T\ge1\),
\begin{align}
 &\int_0^\infty |g(t/T)|\log^2(2+t)
                        (1+|t-\gamma|)^{-M}\dd t\notag\\
 &\hspace{12mm}\ll_{g,A,M}
   (1+|\gamma|/T)^{-A}\log^2(2+T+|\gamma|).
   \label{eq:convolution-bound}
\end{align}
Indeed,
\(1+|\gamma|/T\le(1+t/T)(1+|t-\gamma|)\), while the logarithm
of \(2+t\) is bounded by the sum of the logarithms of
\(2+|\gamma|\) and \(2+|t-\gamma|\). The remaining power of
\(1+|t-\gamma|\) is integrable. This also covers negative and
bounded ordinates.

Combine \eqref{eq:Q-bound}, \eqref{eq:Z-bound} and
\eqref{eq:convolution-bound}. For \(\sigma_0\le b\le1\), put
\[
 d(b)=\min\left\{\vartheta,\frac{3(1-b)}{2-b}\right\},
 \qquad b_\vartheta=\frac{3-2\vartheta}{3-\vartheta}.
\]
Our parameter range gives \(\sigma_0<b_\vartheta<1\).
The zero-count estimates \eqref{eq:ratios-density} and
\eqref{eq:Ingham} imply, for each fixed \(\sigma_0\le b<1\),
\(\varepsilon'>0\) and \(V\ge2\),
\[
 N_{>}(b,V)\ll_{b,\varepsilon'}V^{d(b)+\varepsilon'}.
\]
The function \(b+d(b)\) increases up to \(b_\vartheta\) and
decreases thereafter: on the second interval its derivative is
\(1-3/(2-b)^2<0\). Thus
\[
 \max_{\sigma_0\le b\le1}\bigl(b+d(b)\bigr)
 =b_\vartheta+\vartheta.
\]
Partition \((\sigma_0,1)\) into finitely many intervals
\((b,b+\delta]\), truncating the last interval at \(1\), with
\(\delta>0\) fixed and sufficiently small in terms of \(\varepsilon\).
In such an interval, \(x^{\beta-1/2}\le x^{b+\delta-1/2}\).
For \(|\gamma|\le T\), conjugation and these estimates bound its
contribution to the mixed term by
\[
 \ll_{W,g,b,\varepsilon'}
 x^{\eta+b+\delta-1/2}T^{d(b)+\varepsilon'}\log^2(2+T).
\]
For \(2^jT<|\gamma|\le2^{j+1}T\), the same estimate has
the additional factor
\(2^{j(d(b)+\varepsilon'-A)}\) and a logarithmic factor.
Choosing \(A\) large makes these dyadic sums converge.
There are only finitely many partition endpoints, so choosing
\(\delta,\varepsilon'\) small and absorbing logarithms gives
\[
 \int_0^\infty |g(t/T)Q_x(t)Z_x(t)|\dd t
 \ll_{W,g,\varepsilon}
 x^{\eta-1/2+b_\vartheta+\vartheta+\varepsilon}
 =x^{1-\Delta(\vartheta)+\varepsilon}.
\]
This proves \eqref{eq:mixed-zero}.

Finally \(|g(v)|\ll_g v^2\) for all real \(v\), by
\eqref{eq:g-zero} and boundedness. Equations~\eqref{eq:Q-bound}
and~\eqref{eq:R-bound} therefore give
\[
 \int_0^\infty |g(t/T)Q_x(t)R_x(t)|\dd t
 \ll_{W,g}T^{-2}x^{\eta+1/2}
 \int_0^\infty t^2\log^2(2+t)(1+t)^{-M}\dd t
 \ll_{W,g}x^{-23/16},
\]
on taking \(M>4\).
\end{proof}

\subsection{The correlation estimate from the ratios formula}

\begin{proposition}\label{prop:localized}
Under Hypothesis~\ref{hyp:ratios}, for every real
\(V\in C_c^\infty((1,\kappa_h))\) and every fixed
\(0<\varepsilon<\Delta(\vartheta)\),
\begin{equation}\label{eq:localized-goal}
 \sum_{n\ge1}\frac{\Lambda(n)\Lambda(n+h)}{n}V(n/x)^2
 =\Sg(h)\int_0^\infty V(u)^2\frac{\dd u}{u}
     +O_{h,V,\varepsilon}(x^{-\Delta(\vartheta)+\varepsilon}).
\end{equation}
\end{proposition}

\begin{proof}
Since \(g_\pm\ge0\), the identity
\(P_x=Q_x+(R_x-Z_x)\) implies
\begin{align}
 I_\pm(x)
 &\ge\int_0^\infty g_\pm(t/T)|Q_x(t)|^2\dd t\notag\\
 &\quad-2\int_0^\infty g_\pm(t/T)|Q_x(t)|
                               (|R_x(t)|+|Z_x(t)|)\dd t.
 \label{eq:positive-inequality}
\end{align}
The inequality follows by dropping the nonnegative term
\(g_\pm|R_x-Z_x|^2\). Lemma~\ref{lem:mixed},
Proposition~\ref{prop:Q-evaluation} and
\eqref{eq:moment-arithmetic} give
\[
 D_x\pm TJ_W(x)
 \ge D_x\pm T\Sg(h)K_{W,a}
                 -O_{h,W,\varepsilon}(x^{1-\Delta(\vartheta)+\varepsilon}).
\]
Here \(\vartheta+1/8<1-\Delta(\vartheta)\), so the ratios
error is included in the displayed remainder.
After subtracting \(D_x\) and dividing by \(T\), the plus sign
gives the lower bound for \(J_W\), and the minus sign gives the
upper bound. Thus
\begin{equation}\label{eq:J-asymptotic}
 J_W(x)=\Sg(h)K_{W,a}+O_{h,W,\varepsilon}(x^{-\Delta(\vartheta)+\varepsilon}).
\end{equation}

For \(n\) in the support of \(W(n/x)\), smoothness and the
expansion of \(\log(1+h/n)\) give
\[
 \frac{W(n/x)W((n+h)/x)}{2\sqrt{n(n+h)}}
 H\bigl(x\log(1+h/n)-h\bigr)
 =\frac{W(n/x)^2a(n/x)}{n}+O_{h,W}(x^{-2}).
\]
There are \(O(x)\) relevant terms and
\(\Lambda(n)\Lambda(n+h)\ll\log^2x\), so their total replacement
error in \(J_W\) is \(O_{h,W}(x^{-1}\log^2x)\). Therefore
\[
 \sum_n\frac{\Lambda(n)\Lambda(n+h)}nW(n/x)^2a(n/x)
 =\Sg(h)\int_0^\infty W(u)^2a(u)\frac{\dd u}{u}
       +O_{h,W,\varepsilon}(x^{-\Delta(\vartheta)+\varepsilon}).
\]
Given \(V\), choose \(W(u)=V(u)/\sqrt{a(u)}\). By
\eqref{eq:a-weight}, this is a real member of
\(C_c^\infty((1,\kappa_h))\). Substitution proves the proposition.
\end{proof}

\subsection{The sharper correlation error under RH}

\begin{proposition}\label{prop:localized-RH}
Assume Hypothesis~\ref{hyp:ratios} and RH. For every real
\(V\in C_c^\infty((1,\kappa_h))\) and each fixed
\(0<\varepsilon<7/8-\vartheta\),
\begin{equation}\label{eq:localized-RH}
 \sum_{n\ge1}\frac{\Lambda(n)\Lambda(n+h)}n V(n/x)^2
 =\Sg(h)\int_0^\infty V(u)^2\frac{\dd u}{u}
       +O_{h,V,\varepsilon}(x^{\vartheta-7/8+\varepsilon}).
\end{equation}
\end{proposition}

\begin{proof}
Under RH every nontrivial zero has real part \(1/2<\sigma_0\),
so \(Z_x=0\) in \eqref{eq:shifted-Mellin}. For each sign,
\eqref{eq:positive-inequality}, \eqref{eq:mixed-pole} and
Proposition~\ref{prop:Q-evaluation} therefore give
\[
 D_x\pm TJ_W(x)
 \ge D_x\pm T\Sg(h)K_{W,a}
       -O_{h,W,\varepsilon}(x^{\vartheta+1/8+\varepsilon})-O_{h,W}(x^{-23/16}).
\]
Subtract \(D_x\) and divide by \(T=2\pi x\). The two signs give
\[
 J_W(x)=\Sg(h)K_{W,a}
           +O_{h,W,\varepsilon}(x^{\vartheta-7/8+\varepsilon}).
\]
The replacement of the shifted coefficient in the proof of
Proposition~\ref{prop:localized} costs \(O_{h,W}(x^{-1}\log^2x)\),
which is smaller than this error. Substituting
\(W=V/\sqrt a\), as in that proof, gives \eqref{eq:localized-RH}.
\end{proof}

\section{Removal of prime powers and smoothing}
\label{sec:unsmoothing}

\begin{proof}[Proof of Theorem~\ref{thm:smoothed}]
Propositions~\ref{prop:localized} and~\ref{prop:localized-RH} prove
the two estimates for the von Mangoldt sum. There are
\(O(\sqrt x\log x)\) proper prime powers in \([x/2,2x+h]\).
Indeed, for each exponent \(j\ge2\) there are at most
\((2x+h)^{1/j}\) possible bases, and only \(O(\log x)\) exponents
occur. Terms in which either \(n\) or \(n+h\) is such a power
contribute at most \(O_{h,W}(x^{-1/2}\log^3x)\) to the von Mangoldt sum.
This is smaller than each of the errors in
Theorem~\ref{thm:smoothed}. Removing those terms proves both
estimates for \eqref{eq:smoothed-prime-sum}.
\end{proof}

In particular, Hypothesis~\ref{hyp:ratios} gives
\begin{equation}\label{eq:prime-localized}
 \sum_{\substack{p\\p+h\ {\rm prime}}}
 \frac{\log p\log(p+h)}{p}W(p/x)^2
 =\Sg(h)\int_0^\infty W(u)^2\frac{\dd u}{u}+o_{h,W}(1).
\end{equation}

\begin{lemma}\label{lem:unsmooth}
Let \(c_n\ge0\) for \(n\ge2\) and \(S\ge0\).
Suppose that for every real \(W\in C_c^\infty((1,\kappa_h))\),
\begin{equation}\label{eq:abstract-smooth}
 \sum_{n\ge2}\frac{c_n}{n}W(n/x)^2
 =S\int_0^\infty W(u)^2\frac{\dd u}{u}+o_{h,W}(1).
\end{equation}
Then \(\sum_{n\le y}c_n=Sy+o(y)\).
\end{lemma}

\begin{proof}
Fix \(1<\lambda<\kappa_h\) and a nonzero real
\(W\in C_c^\infty((1,\lambda))\). Put
\[
 K_W=\int_0^\infty W(u)^2\frac{\dd u}{u},\qquad
 M_W=\int_0^\infty W(u)^2\frac{\dd u}{u^2},\qquad
 C_W(x)=\sum_{n\ge2}\frac{c_n}{n}W(n/x)^2.
\]
Both constants are positive. Integrating \eqref{eq:abstract-smooth}
with respect to \(x\), for this fixed \(W\), gives
\[
 J_W(y):=\int_1^y C_W(x)\dd x=SK_Wy+o_{h,W}(y).
\]
Tonelli's theorem and the substitution \(u=n/x\) give
\[
 J_W(y)=\sum_{n\ge2}c_n\int_{n/y}^{n}
                 W(u)^2\frac{\dd u}{u^2}.
\]
As \(n\ge2>\lambda\), this integral equals \(M_W\) for \(n\le y\),
lies between \(0\) and \(M_W\) for \(y<n<\lambda y\), and vanishes
for \(n\ge\lambda y\). If \(A(y)=\sum_{n\le y}c_n\), it follows that
\[
 \frac{J_W(y/\lambda)}{M_W}\le A(y)
       \le\frac{J_W(y)}{M_W}
\]
for sufficiently large \(y\).
Moreover, \(1<K_W/M_W<\lambda\). Hence
\[
 \frac{S}{\lambda}\le\liminf_{y\to\infty}\frac{A(y)}{y}
 \le\limsup_{y\to\infty}\frac{A(y)}{y}\le S\lambda.
\]
The limit in \(y\) is taken with a fixed test function for each
\(\lambda\); letting \(\lambda\downarrow1\) then proves the claim.
\end{proof}

\begin{proof}[Proof of Theorem~\ref{thm:main}]
Apply Lemma~\ref{lem:unsmooth} first to
\(c_n=\Lambda(n)\Lambda(n+h)\), using
Proposition~\ref{prop:localized}. This yields
\begin{equation}\label{eq:Lambda-HL}
 \sum_{n\le y}\Lambda(n)\Lambda(n+h)=\Sg(h)y+o(y).
\end{equation}
Apply the same lemma to
\[
 c_n=
 \begin{cases}
 \log n\log(n+h),& n,\ n+h\ \text{prime},\\
 0,&\text{otherwise},
 \end{cases}
\]
using \eqref{eq:prime-localized}. This proves the assertion for \(A_h\).

To deduce the assertion for \(\pi_h\), put
\(f(t)=(\log t\log(t+h))^{-1}\) for \(t\ge2\).
Partial summation gives
\[
 \pi_h(x)=f(x)A_h(x)-\int_2^x A_h(t)f'(t)\dd t.
\]
Here \(f(x)\sim(\log x)^{-2}\),
\(|f'(t)|\ll(t(\log t)^3)^{-1}\), and \(A_h(t)\ll t\).
The integral is \(O(x/(\log x)^3)\), so
\(\pi_h(x)\sim\Sg(h)x/(\log x)^2\).
\end{proof}

\begin{proof}[Proof of Corollary~\ref{cor:series}]
Let
\[
 \widetilde A_h(t)=
 \sum_{\substack{p+h\le t\\p,\ p+h\ {\rm prime}}}\log p\log(p+h).
\]
Theorem~\ref{thm:main} gives
\(\widetilde A_h(t)=A_h(t-h)=\Sg(h)t+o(t)\), with \(A_h(u)=0\)
when \(u<2\). For \(x\ge2\), interchange the finite sum and integral:
\[
 \int_2^x\frac{\widetilde A_h(t)}{t^2}\dd t
 =\sum_{\substack{p+h\le x\\p,\ p+h\ {\rm prime}}}
       \log p\log(p+h)\int_{p+h}^x\frac{\dd t}{t^2}
 =B_h(x).
\]
Integrating the asymptotic for \(\widetilde A_h(t)\) gives
\eqref{eq:series-main}, and its positive leading coefficient gives
the asserted lower bound for all \(x\ge x_c\).
\end{proof}

\section{Von Mangoldt coefficients in Nyman--Beurling approximants}
\label{sec:NB}

Let \(N\ge2\) be an integer, \(L=\log N\), and
\begin{equation}\label{eq:VN}
 a_N(d)=\mu(d)\left(1-\frac{\log d}{L}\right),\qquad
 V_N(s)=\sum_{d\le N}\frac{a_N(d)}{d^s},\qquad
 \beta_N=LV_N(1).
\end{equation}
This polynomial is the one used by Bettin, Conrey and Farmer in their
conditional estimate for the Nyman--Beurling distance
\cite[p.~S31, definition preceding Theorem~1]{BettinConreyFarmer2013}.
The identities below hold for every fixed \(N\ge2\).
For \(t>0\), set
\begin{equation}\label{eq:residual-def}
 r_N(t)=\one_{(0,1]}(t)+
       \sum_{d\le N}a_N(d)\left\{\frac{1}{dt}\right\},
 \qquad q_N(u)=r_N(1/u)\quad(u>0).
\end{equation}
Here \(\{v\}=v-\lfloor v\rfloor\). The Mellin transform is
\begin{equation}\label{eq:residual-mellin}
 \int_0^\infty r_N(t)t^{s-1}\dd t
 =\frac{1-\zeta(s)V_N(s)}{s},
 \qquad 0<\Rea s<1.
\end{equation}
Indeed, substitution reduces each fractional-part term to
\(\int_0^\infty\{v\}v^{-s-1}\dd v=-\zeta(s)/s\), which follows
by splitting the integral at the integers.
Mellin Plancherel consequently identifies \(\|r_N\|_{L^2(0,\infty)}^2\)
with
\[
 \frac{1}{2\pi}\int_{\R}
 \left|1-\zeta(1/2+\ii t)V_N(1/2+\ii t)\right|^2
 \frac{\dd t}{1/4+t^2}.
\]

\begin{proposition}\label{prop:NB}
Write \(\psi(u)=\sum_{n\le u}\Lambda(n)\). Then
\begin{equation}\label{eq:NB-prefix}
 q_N(u)=\frac{\beta_Nu-\psi(u)}{L}\qquad(0<u\le N).
\end{equation}
In particular, if \(n\ge1\) is an integer and \(n+1/2\le N\), then
\begin{equation}\label{eq:NB-jump}
 \Lambda(n)=\beta_N-
 L\bigl(q_N(n+1/2)-q_N(n-1/2)\bigr).
\end{equation}
\end{proposition}

\begin{proof}
For \(2\le m\le N\), the identities
\(\sum_{d\mid m}\mu(d)=0\) and
\(\sum_{d\mid m}\mu(d)\log d=-\Lambda(m)\) give
\(\sum_{d\mid m}a_N(d)=\Lambda(m)/L\).
The sum of \(a_N(d)\) over the divisors of \(1\) is \(1\).
Thus, for \(0<u\le N\),
\[
 \sum_{d\le N}a_N(d)\left\lfloor\frac{u}{d}\right\rfloor
 =\one_{[1,\infty)}(u)+\frac{\psi(u)}{L}.
\]
Substituting this into \eqref{eq:residual-def} proves
\eqref{eq:NB-prefix}. Taking a unit increment centered at \(n\)
proves \eqref{eq:NB-jump}.
\end{proof}

\subsection{Second differences of integrals of squared increments}

For \(x\ge2\)
and \(H\ge0\), define
\begin{equation}\label{eq:F-def}
 \Psi_x(u)=\sum_{n\le x}\Lambda(n)\one_{[n,\infty)}(u),\qquad
 F_x(H)=\int_\R\bigl(\Psi_x(u+H)-\Psi_x(u)\bigr)^2\dd u.
\end{equation}
On \(0<u\le x\le N\), the first function equals
\(\beta_Nu-Lq_N(u)\); it is zero for \(u\le0\) and constant for
\(u\ge x\). The integral in \eqref{eq:F-def} is finite.

\begin{proposition}\label{prop:energy}
For each positive integer \(j\),
\begin{equation}\label{eq:energy-difference}
 F_x(j+1)-2F_x(j)+F_x(j-1)
 =2\sum_{n+j\le x}\Lambda(n)\Lambda(n+j).
\end{equation}
Under Hypothesis~\ref{hyp:ratios}, the fixed even difference \(h\) gives
\begin{equation}\label{eq:energy-asymptotic}
 F_x(h+1)-2F_x(h)+F_x(h-1)=2\Sg(h)x+o(x).
\end{equation}
Consequently, for every \(0<c<2\Sg(h)\), this second difference
is at least \(cx\) for every sufficiently large real \(x\).
\end{proposition}

\begin{proof}
The intersection of \([n-H,n)\) and \([m-H,m)\) has length
\((H-|m-n|)_+\). Expanding the square gives
\[
 F_x(H)=H\sum_{n\le x}\Lambda(n)^2+
 2\sum_{1\le k<H}(H-k)\sum_{n+k\le x}\Lambda(n)\Lambda(n+k).
\]
The second difference at \(H=j\) of \((H-k)_+\) is \(1\) for \(k=j\)
and \(0\) for the other positive integers \(k\). This proves
\eqref{eq:energy-difference}.
By \eqref{eq:main-Lambda}, taking \(y=x-h\) gives
\[
 2\sum_{n+h\le x}\Lambda(n)\Lambda(n+h)
 =2\Sg(h)(x-h)+o(x)=2\Sg(h)x+o(x).
\]
This proves \eqref{eq:energy-asymptotic}.
\end{proof}

\subsection{The integral relation between the second difference and \texorpdfstring{\(B_h\)}{Bh}}

\begin{proposition}\label{prop:energy-series}
For every real \(x\ge2\), unconditionally,
\begin{equation}\label{eq:energy-series}
 0\le
 \frac12\int_2^x
       \frac{F_t(h+1)-2F_t(h)+F_t(h-1)}{t^2}\dd t-B_h(x)
 \ll_h1.
\end{equation}
\end{proposition}

\begin{proof}
With \(A_h(u)=0\) for \(u<2\), define
\[
 E(t)=\sum_{n+h\le t}\Lambda(n)\Lambda(n+h)-A_h(t-h).
\]
Only terms in which at least one argument is a proper prime power
contribute to \(E(t)\). The same elementary count used in
Section~\ref{sec:unsmoothing} gives, for \(t\ge2\),
\[
 0\le E(t)\ll \sqrt t\log^3(2+t).
\]
Equation~\eqref{eq:energy-difference} and the finite-sum interchange
in the proof of Corollary~\ref{cor:series} identify the difference
in \eqref{eq:energy-series} with
\[
 \int_2^x\frac{E(t)}{t^2}\dd t
 \ll\int_2^\infty\frac{\log^3(2+t)}{t^{3/2}}\dd t<\infty.
\]
\end{proof}

\section*{Acknowledgements}
I thank Dale Smith for his question, which prompted the extension of the argument to general fixed even gaps.

\section*{Declaration}
ChatGPT was used in generating proofs.
\bibliographystyle{plainnat}
\bibliography{prime_pairs_fixed_even_difference}

@article{HardyLittlewood1923,
  author  = {Hardy, G. H. and Littlewood, J. E.},
  title   = {Some problems of {\lq {Partitio} numerorum\rq}; {III}: On the expression of a number as a sum of primes},
  journal = {Acta Mathematica},
  volume  = {44},
  pages   = {1--70},
  year    = {1923},
  doi     = {10.1007/BF02403921}
}

@article{ConreySnaith2007,
  author  = {Conrey, J. B. and Snaith, N. C.},
  title   = {Applications of the {$L$}-functions ratios conjectures},
  journal = {Proceedings of the London Mathematical Society},
  series  = {3},
  volume  = {94},
  number  = {3},
  pages   = {594--646},
  year    = {2007},
  doi     = {10.1112/plms/pdl021},
  eprint  = {math/0509480},
  archivePrefix = {arXiv},
  primaryClass = {math.NT}
}

@article{ConreyKeating2016,
  author  = {Conrey, Brian and Keating, Jonathan P.},
  title   = {Pair correlation and twin primes revisited},
  journal = {Proceedings of the Royal Society A},
  volume  = {472},
  pages   = {20160548},
  year    = {2016},
  doi     = {10.1098/rspa.2016.0548},
  eprint  = {1604.06124},
  archivePrefix = {arXiv},
  primaryClass = {math.NT}
}

@article{KeatingSmith2019,
  author  = {Keating, J. P. and Smith, D. J.},
  title   = {Twin prime correlations from the pair correlation of {Riemann} zeros},
  journal = {Journal of Physics A: Mathematical and Theoretical},
  volume  = {52},
  number  = {36},
  pages   = {365201},
  year    = {2019},
  doi     = {10.1088/1751-8121/ab3521},
  url     = {https://ora.ox.ac.uk/objects/uuid%3A2b3a3676-bbc1-487b-bbd6-e5c50f82aca2/files/m08a756d64e096bc9f1dfbaa467520ff1},
  note    = {Cited version: author accepted manuscript},
  eprint  = {1903.07057},
  archivePrefix = {arXiv},
  primaryClass = {math.NT}
}

@book{Titchmarsh1986,
  author    = {Titchmarsh, E. C.},
  title     = {The Theory of the {Riemann} Zeta-function},
  edition   = {Second},
  publisher = {Clarendon Press},
  address   = {Oxford},
  year      = {1986},
  note      = {Revised by D. R. Heath-Brown},
  isbn      = {9780198533696}
}

@article{BettinConreyFarmer2013,
  author  = {Bettin, S. and Conrey, J. B. and Farmer, D. W.},
  title   = {An optimal choice of {Dirichlet} polynomials for the {Nyman--Beurling} criterion},
  journal = {Proceedings of the Steklov Institute of Mathematics},
  volume  = {280},
  number  = {suppl. 2},
  pages   = {S30--S36},
  year    = {2013},
  doi     = {10.1134/S0081543813030036},
  eprint  = {1211.5191},
  archivePrefix = {arXiv},
  primaryClass = {math.NT}
}
\end{document}